\documentclass[12pt,notitlepage]{article}
\usepackage{ljm-auth}
\usepackage{amsmath}
\usepackage{amssymb}
\usepackage{amscd}
\usepackage{amsfonts}
\usepackage{amsthm}
\usepackage{mathrsfs}
\usepackage[matrix,arrow,curve]{xy}

  \newtheorem{theorem}{Theorem}[section]
  \newtheorem{prop}[theorem]{Proposition}

\theoremstyle{remark}
  \newtheorem{rem}{Remark}

\newcommand{\Dj}{\hbox to 8pt{\raisebox{.4\height}{-}\hss D}}

\newcommand{\ac}{\ensuremath{\mathcal A}}

\newcommand{\hc}{\ensuremath{\mathcal H}}

\newcommand{\uc}{\ensuremath{\mathcal U}}

\newcommand{\xc}{\ensuremath{\mathcal X}}

\newcommand{\g}{\ensuremath{\mathfrak{g}}}

\newcommand\beq{\begin{equation}}
\newcommand\eeq{\end{equation}}
\newcommand\bqa {\begin{eqnarray}}
\newcommand\eqa {\end{eqnarray}}

\newcommand{\bear}{\begin{array}}
\newcommand{\enar}{\end{array}}

\usepackage{graphicx}
\usepackage{wrapfig}

\author{G.~Sharygin
}
\crauthor{G.~Sharygin} 

\tit{Deformation quantization and the action of Poisson vector fields}
\shorttit{Deformation quantization and vector fields} 

\begin{document}

\maketit

\address{Moscow State University, Leninskie gory, GSP-1, 119991 Moscow, Russia}

\email{sharygin@itep.ru}

\Received{28.11.2016}

\abstract{As one knows, for every Poisson manifold $M$ there exists a formal noncommutative deformation of the algebra of functions on it; it is determined in a unique way (up to an equivalence relation) by the given Poisson bivector. Let a Lie algebra $\mathfrak g$ act by derivations on the functions on $M$. The main question, which we shall address in this paper is whether it is possible to lift this action to the derivations on the deformed algebra. It is easy to see, that when dimension of $\mathfrak g$ is $1$, the only necessary and sufficient condition for this is that the given action is by Poisson vector fields. However, when dimension of $\mathfrak g$ is greater than $1$, the previous methods do not work. In this paper we show how one can obtain a series of homological obstructions for this problem, which vanish if there exists the necessary extension.}

\notes{0}{
\thank{This work was supported by the Russian Science Foundation grant 16-11-10069} }

 

\section{Introduction}
We begin with few introductory and motivational remarks, needed to acquaint the reader with the basic ideas and definitions of the theory under consideration. We shall give only brief descriptions of results; interested reader should consult the referenced papers for details. 
\subsection{Deformation theory}
Let \ac\ be a complex commutative algebra, for example, $\ac=C^\infty(M)$ the algebra of $\mathbb C$-valued functions on a smooth manifold $M$. 
One can define its \textit{deformation quantization} as a new associative $\hbar$-linear multiplication $\ast$ on the space of formal power series (with respect to the variable $\hbar$) with coefficients from $\ac$. In case $\ac=C^\infty(M)$ one usually asumes that for any $f,\,g\in\ac$ the product given by a formal power series of \textit{bidifferential operators}:
$$
f\ast g=fg+\sum_{k=1}^\infty\hbar^kB_k(f,g).
$$
Here $fg$ is the usual (commutative) product in \ac. Two deformations $\ast_1,\,\ast_2$ are said to be \textit{equivalent} if there exists a formal power series of differential operators $T$,
$$
T(f)=f+\sum_{k=1}^\infty \hbar^kT_k(f)
$$
for all $f\in\ac$, such that
$$
T(f\ast_1 g)=T(f)\ast_2T(g).
$$
In fact it is easy to show that up to this equivalence one can always choose $B_1(f.g)=\{f,g\}$, where $\{,\}$ denotes a Poisson bracket; in particular, in the case $\ac=C^\infty(M)$ this means that $M$ should bear a Poisson structure. Recall that Poisson structure on a manifold is determined by a Poisson bivector $\pi$, i.e. by a section of the exterior square of the tangent bundle on the manifold, whose Schouten brackets with itself vanish. Given such bivector we define the Poisson bracket on $C^\infty(M)$ as
$$
\{f,\,g\}=\pi(df,dg).
$$
This bracket verifies the Jacobi identity and Leibniz rule with respect to the product of functions.

An important particular case of this construction is given by the Kirillov-Kostant bracket on a coadjoint space $\g^*$ of a Lie algebra $\g$.  Thus, in this case $M=\g^*$ is an affine space, equipped with the following Poisson bracket:
$$
\{f,\,g\}(x) = <[df,\,dg], x>.
$$
Here $f, g \in C^{\infty}(\g^*),\ x \in \g^*$, and we use the angular brackets to denote the natural pairing of dual spaces. Let us fix a basis in $\g$. The Poisson bivector in this case can be expressed as a linear expression in coordinates $x_k$ in $\g^*$ in terms of the structure constants $c^k_{ij}$ of \g: 
$$
\pi = c^k_{ij}x^k\partial_i\otimes\partial_j.
$$

We shall now assume that the first term in the deformation series coincides with the Poisson bivector, i.e. that up to higher degrees in $\hbar$ the commutator of any two functions $f$ and $g$ (with respect to the new product) is equal to their Poisson bracket:
$$
[f,g]=\hbar\{f,g\}+O(\hbar^2).
$$
The main problem of the deformation quantization program is to classify all possible deformation quantizations of an algebra up to the equivalence. This problem has been solved in different ways by different authors and under different assumptions. The principal result here is due to Kontsevich \cite{Kon97,Kon03}, where the main theorem is proved, stating that there always exists a deformation quantization of a Poisson variety $(M,\pi)$; this deformation is given by an explicit formula. It is also proved that the quantizations correspond bijectively to the equivalence classes of formal power series of bivectors $\Pi=\pi+\sum_{k\ge 1}\hbar^k\pi_k$, verifying the following variant of the Jacobi identity
$$
[\Pi,\Pi]=0,
$$ 
where the brackets denote the $\hbar$-linear extension of the Schouten brackets.

In many cases deformation quantization of a Poisson algebra can be described in a rather simple way. For example, when the manifold $M=\mathbb R^{2n}$ with coordinates $p_1,\dots,p_n,q_1,\dots,q_n$ and the Poisson structure is given by a constant symplectic form, which can be written as
$$
\omega=\sum_{k=1}^ndp_k\wedge dq_k,
$$
then the deformation quantization can be chosen to be the Weyl algebra 
$$
\ac^n_\hbar=\mathbb C\langle x_1,\dots,x_n,\partial_1,\dots,\partial_n\rangle/\{[x_k,x_j]=[\partial_k,\partial_j]=0,\,[\partial_k,x_j]=i\hbar\delta_{kj}\},
$$
where $\mathbb C\langle x_1,\dots,x_n,\partial_1,\dots,\partial_n\rangle$ denotes the free algebra. In case, when the Poisson manifold is equal to the coadjoint representation of a semisimple Lie group, endowed with the Kirillov-Kostant symplectic structure, the quantized algebra can be identified with the universal enveloping algebra (see \cite{Kon97}). More generally, if $M=T^*X$ with the usual symplectic structure, deformation quantization is closely related to the algebra of differential operators on $X$ (\cite{Fedosov}). 

\subsection{The main question}
In this paper we deal with the following question: suppose there are additional structures (e.g. symmetries, a group or an algebra action, complex structure, etc.) of the Poisson structure on $M$. When is it possible to transfer them from the ``classical''  to quantum case, to the deformed algebras?

In addition to being interesting on its own right, this question might be quite important for many applications. For example, if there is a commutative Lie subalgebra in the functions (with respect to the Poisson bracket), then on one hand it induces an action of commutative Lie algebra on $M$, which we can try to pull to the deformed algebra; if the latter action is by internal derivations, this would give us a commutative subalgebra in the deformation of the manifold. On the other hand, when $M=\mathbb R^n$, its defomation is given by Weyl algebra and the structure of commutative subalgebras in Weyl algebra is closely related to the famous Jacobian problem (see \cite{Kon-Bel}). Commutative subalgefbras of the differential operators are closely connected to the integrable systems theory (see \cite{Krich}), and are often called ``quantum integrable systems''.

In this paper we shall investigate the following general form of this question: let a Lie algebra $\g$ act on $C^\infty(M)$ by differentiations (i.e. represented in the Lie algebra $Vect(M)$ of vector fields on a Poisson manifold $M$), which preserve the Poisson structure:
$$
\xi(\{f,\,g\})=\{\xi(f),\,g\}+\{f,\,\xi(g)\},
$$
(in the terms of vector fields, one can say, that $L_\xi\pi=0$, where $L_\xi$ is the Lie derivative, and $\pi$ the Poisson bivector); vector fields with this property are called \textit{Poisson fields}. A particular case of Poisson fields is given by the \textit{Hamiltonian fields}, or skew gradients of functions; the easiest way to define the skew gradient $X_f$ of a function $f$ is by the formula
$$
X_f(g)=-\{f,g\}
$$
for all functions $g\in C^\infty(M)$. Now the question is: \textit{is it possible to find an extension of this representation to an action of \g\ on the quantized algebra $(C^\infty(M)[[\hbar]],\,\ast)$ by derivations, so that for any $\xi$ in \g\ the corresponding derivation $\hat\xi$ of $(C^\infty(M)[[\hbar]],\,\ast)$ would have the form $\hat\xi=\xi+o(\hbar)$}?

More accurately, the question is, if one can find a linear map
$$
\g\to Der(C^\infty(M)[[\hbar]]),\ \xi\mapsto\hat\xi,
$$
such that $\hat\xi=\xi+o(\hbar)$ and
$$
\widehat{[\xi,\eta]}=[\hat\xi,\hat\eta].
$$
Observe that if this is possible and when $\g=\mathbb R^n$ is the commutative Lie algebra, induced by an integrable system $f_1,\dots,f_n$, then the quantum integrability is reduced to the question, whether the derivations $\hat\xi$ are inner, or not.

\subsection{Notations, agreements and main results}
In what follows all algebras are considered over fields $\mathbb R$ or $\mathbb C$; $M$ is a smooth compact cosed Poisson manifold, $C^\infty(M)$ its algebra of smooth functions, $\pi$ is the bivector and $\{\,\}$ the Poisson bracket. Unless otherwise stated we shall deal with the generic deformation formula (e.g. Kontseveich's formula, see \cite{Kon03} and section \ref{sec:Kon1}), so that 
$$
f\ast g=fg+\frac12\hbar\{f,g\}+o(\hbar).
$$

In the next section we recall basic facts  and constructions from the theory of Hochschild cohomology, $L_\infty$-algebras etc., which we shall use below. We try to make our exposition as much self-contained as we can. Major references for these sections are Loday's book \cite{Loday92}, and Kontsevich's seminal paper \cite{Kon03}, and references therein (see also a very nice survey of Keller \cite{Keller03}).

We further begin answering the question we posed. First we consider the simplest case, when dimension of the Lie algebra is equal to $1$, so that $\g$ is equal to the linear span $\mathbb R\langle X\rangle$ where $X$ is a Poisson vector field, i.e. $L_X\pi=0$, or 
$$
X(\{f,g\})=\{X(f),g\}+\{f,X(g)\}.
$$
In this case in order to quantize the Lie algebra action, it is enough to find one differentiation $\hat X$ of the quantized algebra, such that $$
\hat X=X+o(\hbar).
$$ 
In section \ref{sec:Kq} we show, that this problem can be solved by an application of Kontsevich's quasi-isomorphism map from \cite{Kon03}, see also \cite{CvdB}. However the same method fails for generic Lie algebras, yielding only an action up to inner derivatives of the algebra $(C^\infty(M),\ast)$. In order to deal with  general situation we develop in further sections, \ref{sec:def1}, \ref{sec:def2} an obstruction theory and show that the quantization exists if certain classes in the cohomology of $\mathfrak g$ with coefficients in Lichnerowicz-Poisson and Hochschild cohomology of $M$ vanish. It follows from our previous considerations that these classes can always be made trivial for just one vector field (however, see remark \ref{rem:partialcase}). In the more general case, when $\mathrm{dim}\,\g>1$, acting in a similar way we obtain a series of obstructions with values in Lie algebra cohomology with coefficients.

\section{Preliminary information and constructions}
This section contains brief outlines of the basic constructions and notions that are used in deformation theory. Interested reader can obtain more information from the papers to which we refer below.
\subsection{Hochschild cohomology and deformations}
\label{sec:Hhdef}
Let $M,\pi$ be a Poisson manifold. Then we shall always assume that the deformation quantization of its functions algebra is given by the following formal series in $\hbar$ (c.f. \cite{Keller03}):
\begin{equation}
\label{eq:defq1}
f\ast g=fg+\frac{\hbar}{2}\{f,g\}+\sum_{k=2}^\infty \hbar^k B_k(f,g),
\end{equation}
for all $f,\,g\in C^\infty(M)$. Here $B_k(f,g)$ are some linear (over numbers) differential operators, applied to $f$ and $g$. The associativity condition
\begin{equation}
\label{eq:assoc1}
(f\ast g)\ast h=f\ast(g\ast h),
\end{equation}
for all $f,\,g,\,h\in C^\infty(M)$ can be expanded as a series of partial differential equations on operators $B_k$. These equations are rather complicated. In order to have a more convenient view on this problem it is better to consider the Hochschild complex version of this equation. 

Recall (the book \cite{Loday92} is the main reference for this subject), that for an algebra $A$, its Hochschild cohomology is defined as the cohomology of the complex
$$
C^*(A)=\bigoplus_{n\ge0}Hom(A^{\otimes n},A),
$$
with differential
$$
\begin{aligned}
\delta\varphi(f_1,\dots,f_{p+1})&=f_1\varphi(f_2,\dots,f_{p+1})+\sum_{i=1}^p(-1)^i\varphi(f_1,\dots,f_if_{i+1},\dots,f_{p+1})\\
                                                &\qquad\qquad\qquad\qquad\quad\!\!+(-1)^{p+1}\varphi(f_1,\dots,f_p)f_{p+1}.
\end{aligned}
$$
Here $\varphi\in C^p(A)=Hom(A^{\otimes p},A)$ and $f_1,\dots,f_{p+1}\in A$ are arbitrary elements. In the important particular case, when $A=C^\infty(M)$ for a smooth manifold $M$, one often reduces this complex to the so-called \textit{local Hochschild cohomology complex} $C_{loc}(C^\infty(M))$, in which the spaces of linear maps $Hom(A^{\otimes n},A)$ are replaced with the spaces of \textit{local cochains}, $Hom_{loc}(A^{\otimes n},A)$, given by the polydifferential operators on functions; recall, that a map $\varphi:A^{\otimes p}\to A$ is called polydifferential operator, if for any $k=1,\dots,p$ and any $f_i\in C^\infty(M),\ i=1,\dots,\widehat k,\dots,p$ (here and elsewhere the hat\ \ $\widehat{}$\ \ over an element of an array means that this element is missing), the map
$$
\varphi_k(f)=\varphi(f_1,\dots,f_{k-1},f,f_{k+1},\dots,f_p):C^\infty(M)\to C^\infty(M)
$$
is a linear (over the field) differential operator.

Unlike the usual Hochschild cohomology of $C^\infty(M)$, its local cohomology can be easily calculated, see for instance \cite{CGdW83}: the resulting theorem, usually called (a cohomological version of) \textit{Hochschild-Kostant-Rosenberg} theorem, says that the following map induces an isomorphism in cohomology
$$
\begin{aligned}
\chi: \Gamma(\Lambda^*TM)&\to C^*_{loc}(C^\infty(M))\\
\chi(\Phi)(f_1,\dots,f_p)&=\frac{1}{p!}\sum_{\sigma\in S_p}(-1)^\sigma\Phi(df_{\sigma(1)},df_{\sigma(2)},\dots,df_{\sigma(p)}),
\end{aligned}
$$
for any polyvector field $\Phi\in\Gamma(\Lambda^pTM)$ and any functions $f_1,\dots,f_p\in C^\infty(M)$; here on the left hand side we use zero differential, and on the right the Hochschild differential $\delta$. In particular local Hochschild cohomology of the algebra $C^\infty(M)$ is equal to the space of polyvector fields: $H^*_{loc}(C^\infty(M))=\Gamma(\Lambda^*TM)$. Below we shall usually omit the adjective local, when speaking about the Hochschild cohomology of the smooth functions on a manifold.

Hochschild complex of an algebra $A$ bears many additional algebraic structures. Two most important of them are the cup-product and Gerstenhaber bracket. The cup-product $C^p(A)\otimes C^q(A)\to C^{p+q}(A)$ is determined by the formula:
$$
(\varphi\cup\psi)(f_1,\dots,f_{p+q})=\varphi(f_1,\dots,f_p)\psi(f_{p+1},\dots,f_{p+q}),
$$
where $\varphi\in C^p(A),\ \psi\in C^q(A)$. This is an associative product; differential $\delta$ verifies the graded Leibniz rule with respect to this product:
$$
\delta(\varphi\cup\psi)=\delta(\varphi)\cup\psi+(-1)^p\varphi\cup\delta(\psi).
$$
The Gerstenhaber bracket (c.f. \cite{Keller03}) is a map $[,]:C^p(A)\otimes C^q(A)\to C^{p+q-1}(A)$, determined by the formula
$$
{}[\varphi,\psi]=\sum_{k=1}^p(-1)^{k(q-1)}\varphi\circ_k\psi-(-1)^{(p-1)(q-1)}\sum_{l=1}^q(-1)^{l(p-1)}\psi\circ_l\varphi,
$$
where the composition maps $\circ_k$ are defined by the formulas:
$$
\varphi\circ_k\psi(f_1,\dots,f_{p+q-1})=\varphi(f_1,\dots,f_{k-1},\psi(f_k,\dots,f_{k+q-1}),f_{k+q},\dots,f_{p+q-1}),
$$
i.e. the value of $\psi$ is substituted as an argument into $\varphi$. The map $[,]$ is skew-symmetric with respect to shifted dimension:
$$
[\varphi,\psi]=-(-1)^{(p-1)(q-1)}[\psi,\varphi]
$$
and direct computations show that it verifies the graded Jacobi identity
$$
[\varphi,[\psi,\omega]]=[[\varphi,\psi],\omega]+(-1)^{(p-1)(q-1)}[\psi,[\varphi,\omega]],
$$
or, in more symmetric form:
$$
(-1)^{(p-1)(r-1)}[\varphi,[\psi,\omega]]+(-1)^{(q-1)(p-1)}[\psi,[\omega,\varphi]]+(-1)^{(r-1)(q-1)}[\omega,[\varphi,\psi]]=0
$$
for any $\omega\in C^r(A)$. Observe, that if $\mu:A\otimes A\to A$ is the product map, i.e. if we regard the product in $A$ as an element in $C^2(A)$, then one can define the differential $\delta$ by the formula
$$
\delta(\varphi)=-[\mu,\varphi],
$$
thus it follows from the Jacobi identity that a skew-symmetric version of Leibniz rule holds for $\delta$ with respect to the bracket $[,]$:
$$
\delta[\varphi,\psi]=[\delta\varphi,\psi]+(-1)^{p-1}[\varphi,\delta\psi].
$$
It is clear, that these two operations preserve the space of local cochains. One can now improve the statement of the Hochschild-Kostant-Rosenberg theorem as follows: \textit{the product and the bracket in (local) Hochschild cohomology of the algebra $C^\infty(M)$, induced from the $\cup$-product and the bracket $[,]$ on the local complex coincide with the wedge-product and the Schouten bracket on polyvector fields respectively}. Recall, that the Schouten bracket is the unique bracket on the space of polyvector fields, that verifies the Leibniz rule with respect to wedge product and is given by the commutator on usual vector fields.

It is now easy to write down the conditions, that guarantee the associativity of the $\ast$-product in terms of the operations in Hochschild complex: first of all we interpret the bidifferential operators $B_k$ as elements in $C^2(C^\infty(M))$. To make our notation shorter we shall also put $B_1(f,g)=\frac12\{f,g\}=\chi(\pi)$ (here $\chi$ denotes the Hochschild-Kostant-Rosenberg antisymmetrization map, see above). Now by comparing the coefficients with the same power in $\hbar$ on both sides of \eqref{eq:assoc1} and using the definitions of Gerstenhaber bracket and its properties, listed above, one obtains the equations, that ensure associativity of the star-product. The first few equations are
\begin{equation}
\label{eq:MCeq1}
\begin{aligned}
\delta B_1&=0;\\
\delta B_2&=-\frac12([B_1,B_1])\\
\delta B_3&=-[B_1,B_2]=-\frac12([B_1,B_2]+[B_2,B_1]),
\end{aligned}
\end{equation}
and so on. If we consider the formal power series $B=\sum_{k=1}^\infty \hbar^kB_k$ as an element in $C^*(C^\infty(M))[[\hbar]]$ and extend all the operations in Hochschild complex to this module in an evident way ($\hbar$-linearly), then we can write all these equalities in a rather concise form:
\begin{equation}
\label{eq:MCeq2}
\delta B-\frac12[B,B]=0.
\end{equation}
This equation is usually called \textit{the Maurer-Cartan equation}, see section \ref{sec:Kon1}. The existence (and uniqueness up to an equivalence) of solution of the Maurer-Cartan equation \eqref{eq:MCeq2} with the first term $B_1$ given by a Poisson bivector (i.e. $B_1(f,g)=\frac12\{f,g\}=\chi(\pi)(f,g)$), is guaranteed by the well-known Kontsevich's formality theorem, see \cite{Kon97,Kon03}. In what follows we shall assume, that such a solution $B$ is fixed and denote by $B_k$ its coefficients. Some details on the proof of this theorem can be found in the next section.

\subsection{Kontsevich's map and deformation theory}
\label{sec:Kon1}
A Lie algebra structure on a graded space in which all identities hold with the signs, given by Koszul's sign convention, is called \textit{graded Lie algebra}; if in addition there is a degree $+1$ differential on this space, which verifies the Leibniz rule with respect to the bracket, then it is \textit{differential graded}, or just \textit{differential Lie algebra} (DGLA for short). Cohomology of such algebras inherits a Lie algebra structure. A homomorphism $f:\mathfrak g\to\mathfrak h$ of two differential Lie algebras is called \textit{quasi-isomorphism} if it induces an isomorphism in cohomology. Unlike usual isomorphisms, quasi-isomorphisms do not have inverse homomorphisms. Instead, if we want to find inverse of a quasi-isomorphism, we need to embed the map in a larger category, that of $L_\infty$-algebras and $L_\infty$-homomorphisms between them. Without going deep into details, let us say, that \textit{every differential Lie algebra is an $L_\infty$-algebra, and every homomorphism of Lie algebras is an $L_\infty$-morphism}. It is also possible to give an explicit definition of an $L_\infty$-morphisms between two Lie algebras, $\mathfrak g_1,\,\mathfrak g_2$ with differentials $d_1,d_2$ and brackets $[,]_1,\,[,]_2$. First, we observe that the differential $d_1$ can be extended to the exterior powers of $\mathfrak g_1$; we shall denote this extension by the same symbol $d_1$. Further, we recall, that for any homogeneous map $f:C\to D$ between two (co)chain complexes with differentials $d_1,\,d_2$, its differential is given by 
$$
d(f)=f\circ d_1-(-1)^{|f|}d_2\circ f,
$$
where $|f|$ denotes the homogeneity degree of $f$. Finally, one says, that an $L_\infty$-morphism $F$ from $\mathfrak g_1$ to $\mathfrak g_2$ is given, if there is a collection of maps $F_n:\Lambda^n\mathfrak g_1\to\mathfrak g_2$ of degrees $1-n$, which verify the following sequence of equations (here we omit the signs of $\wedge$-product):
$$
\begin{aligned}
dF_{n+1}&(X_1,X_2,\dots,X_{n+1})=\sum_{1\le i<j\le n+1}(-1)^{\epsilon(i,j)}F_n([X_i,X_j]_1,X_1,\dots,\widehat{X_i},\dots,\widehat{X_j},\dots,X_{n+1})\\
                                          &\quad+\frac12\sum_{i=1}^{n}\sum_{\sigma\in S_{n+1}}\frac{(-1)^{\sigma(X)}}{i!(n-i+1)!}[F_i(X_{\sigma(1)},\dots,X_{\sigma(i)}),F_{n-i+1}(X_{\sigma(i+1)},\dots,X_{\sigma(n+1)})]_2.
\end{aligned}
$$
Here as usually\ \ $\widehat{}$\ \ denotes the missing element, $S_{n+1}$ is the group of permutations in $n+1$ elements and the signs are obtained from Koszul sign rules and depend on the degrees of the elements $X_k$. In particular, this equality shows that $d(F_1)=0$, i.e. $F_1:\mathfrak g_1\to\mathfrak g_2$ is a chain map; further, although the map $F_1$ needs not be a homomorphism of Lie algebras, the second equation shows that $F_2$ is a homotopy, which makes the induced map $F^*_1$ on cohomology a homomorphism. Also observe that every homomorphism $f:\mathfrak g_1\to\mathfrak g_2$ of Lie algebras induces an $L_\infty$-morphism: just put $F_1=f$ and $F_k=0,\,k\ge2$.

One says, that an $L_\infty$-map $F$ is quasi-isomorphism, if $F_1$ induces an isomorphism in cohomology. One can show, that in this case there always exists a homotopy-inverse $L_\infty$-map $G:\mathfrak g_2\to\mathfrak g_1$; thus it is almost as good as an isomorphism of Lie algebras. In addition, almost like the usual homomorphism of Lie algebras, $L_\infty$-morphism $F=\{F_n\}:\mathfrak g_1\to\mathfrak g_2$ allows one transfer algebraic structures from left to right. The structure, that we care for most of all is the set of solutions of the Maurer-Cartan equation. Recall, that an element $\omega$ in a differential graded Lie algebra $\mathfrak g$ such that $\mathrm{deg}\,\omega=1$ is called \textit{a solution of MC equation}, if
$$
d\omega-\frac12[\omega,\omega]=0.
$$
As we have explained earlier, this equation is closely related to the deformation quantization of an algebra; the set of all solutionos of this equation in $\mathfrak g$ is often denoted by $MC(\mathfrak g)$

It is clear, that if $f:\mathfrak g_1\to\mathfrak g_2$ is a homomorphism of differential Lie algebras and $\omega_1\in MC(\mathfrak g_1)$ then $f(\omega_1)$ is in $MC(\mathfrak g_2)$. It turns out, that similar statement is true in case, when there is only an $L_\infty$-morphism $F$ between the algebras: just put
\begin{equation}
\label{eq:trMC1}
F(\omega)=\sum_{n\ge0}\frac{1}{n!}F_n(\underbrace{\omega_1,\dots,\omega_1}_{n\ \mbox{times}}).
\end{equation}
Of course, we must assume, that the sum on the right converges in one or another sense.

Another important construction, closely related with the previous one, is based on the following observation: every solution $\omega$ of the Maurer-Cartan equation in $\mathfrak g$ gives rise to a new differential in $\mathfrak g$: put
$$
d_\omega x=dx+[\omega,x].
$$
Then an easy computation shows that $d_\omega^2=0$:
$$
\begin{aligned}
d_\omega^2 x&=d_\omega(dx+[\omega,x])=d^2x+d[\omega,x]+[\omega,dx]+[\omega,[\omega,x]]\\
                    &=[d\omega,x]-[\omega,dx]+[\omega,dx]-\frac12[[\omega,\omega],x]\\
                    &=[d\omega-\frac12[\omega,\omega],x]=0.
\end{aligned}
$$
Here we have used the Leibniz rule, which holds for the differential $d$ with respect to the Lie brackets, and the Jacobi identity, which in this case reads as:
$$
[\omega,[\omega,x]]=[[\omega,\omega],x]-[\omega,[\omega,x]].
$$
In fact, this new differential commutes with the Lie algebra structure:
$$
\begin{aligned}
d_\omega[x,y]&=d[x,y]+[\omega,[x,y]]=[dx,y]+(-1)^{|x|}[x,dy]+[[\omega,x],y]+(-1)^{|x|}[x,[\omega,y]]\\
                       &=[d_\omega x,y]+(-1)^{|x|}[x,d_\omega y]
\end{aligned}
$$
so that $(\mathfrak g,d_\omega)$ is a differential Lie algebra again.

It turns out, that the $L_\infty$-map $F$ not only allows one transfer the solutions of Maurer-Cartan equation, but also it gives a map $(\mathfrak g_1,d_\omega)\to(\mathfrak g_2,d_{F(\omega)})$; namely put
\begin{equation}
\label{eq:defmap1}
F_\omega(\alpha)=\sum_{k\ge1}\frac{1}{(k-1)!}F_k(\alpha,\underbrace{\omega,\dots,\omega}_{k-1\ \mbox{times}})
\end{equation}
for any $\alpha\in\mathfrak g_1$ (once again we assume, that the sum on the right converges in some sense). Then it is easy to show that $F_\omega$ commutes with the differentials:
$$
F_\omega(d_\omega\alpha)=d_{F(\omega)}F_\omega(\alpha).
$$
Moreover, the map $F_\omega$ can be extended to an $L_\infty$-morphism $F_\omega:(\mathfrak g_1,d_\omega,[,]_1)\to(\mathfrak g_1,d_{F(\omega)},[,]_2)$; in case $F$ was an $L_\infty$-quasiisomorphism, the map $F_\omega$ is also a quasi-isomorphism.

The main result of Kontsevich's paper \cite{Kon97} can be interpreted in the terms of this general theory: consider the differential graded Lie algebras $\mathfrak g_1=\Lambda^*TM[[\hbar]]$ (with zero differential and $\hbar$-linear Schouten bracket) and $\mathfrak g_2=C^*(C^\infty(M))[[\hbar]]$ (with Hochschild differential and $\hbar$-linear Gerstenhaber bracket). Then (see \cite{Kon97}) \textit{there exists an $L_\infty$-quasi-isomorphism $F=\{F_n\}:\mathfrak g_1\to\mathfrak g_2$, such that $F_1=\chi$ is the Hochschild-Kostant-Rosenberg map.} 

Observe, that Poisson bivector $\hbar\pi$ verifies the Maurer-Cartan equation. On the other hand in order to define the $\ast$-product we need an element $B\in MC(\mathfrak g_2)$. Now one can do this by the virtue of formula \eqref{eq:trMC1}, since the convergence is guaranteed by the growing powers of $\hbar$.

\section{$L_\infty$-maps and quantization of Lie algebra actions}
In this section we address the question of how the action of Poisson vector fields on the manifold can be extended to an action of the same fields by derivations on the quantized algebra. It turns out, that there always exist a way to extend the action of just one field, while in case of a multiple independent fields the method we use fails; in fact it only gives an action of the Lie algebra by derivations up to internal differentiations. After this we develop a theory of obstructions, that govern the question. It follows from the previous observation, that in the case of $1$-dimensional Lie algebra, these obstructions can be made equal to $0$ (see however remark \ref{rem:partialcase}), but in a generic case their values are not clear.

\subsection{Quantization of a vector field}
\label{sec:Kq}
Let us begin by showing that any Poisson vector field on $M$ can be extended to a differentiation of the deformed algebra. To this end we shall use Kontsevich's quasi-isomorphism of the algebras of Hochschild cochains and polyvector fields, whose basic definitions and related results we briefly recalled in previous section.

First of all, consider a generic vector field $\xi$ on $M$; it is our purpose to find a deformation of $X$,
\begin{equation}
\label{eq:xi1}
\Xi=\xi+\hbar\xi_1+\hbar^2\xi_2+\dots,
\end{equation}
where $\xi_k$ are differential operators n $M$, so that it induces a differentiation on the deformed algebra $(C^\infty(M)[[\hbar]],\ast)$ (as one can see, we assume that the first term is equal to $\xi$, so that $\Xi=\xi+o(\hbar)$). Let us write down the conditions, which follow from the assumption, that $\Xi$ is a derivation of the noncommutative algebra $(C^\infty(M)[[\hbar]],\ast)$; we can do it by considering one by one the coefficients at different powers of $\hbar$.

First of all, in degree $0$ we have the equation
$$
\xi(ab)=a\xi(b)+\xi(a)b,
$$
i.e. the map $\xi$ should be a derivation of $C^\infty(M)$, which certainly holds, since $\xi$ is a vector field. Further, at degree $1$, we obtain the equality
$$
\xi_1(ab)-a\xi_1(b)-\xi(a)b=\frac12(\{\xi(a),b\}+\{a,\xi(b)\}-\xi(\{a,b\})).
$$
Now this equality can only hold, when entities on both sides vanish: indeed, since $ab=ba$, the left hand side is symmetric in $a$ and $b$, while on the left we have an antisymmetric expression. This means, that the condition that $\xi$ is a Poisson vector field cannot be removed. It turns out, that this condition is sufficient.
\begin{prop}
For any Poisson vector field $\xi$ there exists a continuation $\Xi$ of the form \eqref{eq:xi1}, which is a differentiation of the deformed algebra.
\end{prop}
\begin{proof}
Using the notation from section \ref{sec:Hhdef}, we can write all the relations on $\xi_k$ in the following brief form:
$$
\delta\Xi-[B,\Xi]=0.
$$
In other words, $\Xi$ is a differentiation of $(C^\infty(M)[[\hbar]],\ast)$, if and only if it is closed with respect to the $B$-deformed differential in the Hochschild complex, see section \ref{sec:Kon1}. On the other hand, since $\xi$ is Poisson vector field, it represents a closed element in the $\pi$-deformed differential in polyvector fields (where the original differential is equal to $0$). Now the claim follows directly from the observation in the end of section \ref{sec:Kon1}, see formula \eqref{eq:defmap1}. In this case the corresponding formula takes the form
$$
\Xi=\sum_{n\ge1}\frac{\hbar^{n-1}}{(n-1)!}\,\uc_n(\xi,\pi,\dots,\pi),
$$
where $\uc_n$ is the $n$-th stage of Kontsevich's $A_\infty$-quasi-isomorphism.
\end{proof}
It is instructive to look at the map $\uc_\pi$ (formula \eqref{eq:defmap1}), induced by Kontsevich's quasi-isomorphism when the dimension of \g\ is greater than $1$ (of course we assume that \g\ acts on $M$ by Poisson fields). In this case every element $X\in\g$ induces a differentiation on the deformed algebra, however this map is not a homomorphism of Lie algebras. Indeed, from the definition of $A_\infty$-maps we obtain the following equality, since $[X,\pi]=[Y,\pi]=[\pi,\pi]=0$ in the algebra of polyvector fields
$$
\begin{aligned}
\delta(\uc_n)(X,Y,\pi,\dots,\pi)&=\uc_{n-1}([X,\,Y],\pi,\dots,\pi)\\
                  &\quad -\sum_{p+q=n}\Bigl(\frac{(n-2)!}{(p-1)!(q-1)!}[\uc_p(X,\pi,\dots,\pi),\uc_q(Y,\pi,\dots,\pi)]\\
                   &\qquad\qquad-\frac{(n-2)!}{(p-2)!q!}[\uc_p(X,Y,\pi,\dots,\pi),\uc_q(\pi,\dots,\pi)]\Bigr).
\end{aligned}
$$
Also observe, that due to the dimensional restrictions $\uc_n(X,Y,\pi,\dots,\pi)\in C^0(C^\infty(M))$, so as $C^\infty(M)$ is commutative, $\delta(\uc_n)(X,Y,\pi,\dots,\pi)=0$. Thus we have the following equation (here we use the fact, that on elements of degree $1$, Gerstenhaber brackets coincide with usual commutators)
$$
\begin{aligned}
{}[\uc_\pi(X),\uc_\pi(Y)]&=\sum_{p,q=1}^\infty\frac{\hbar^{p+q-2}}{(p-1)!(q-1)!}[\uc_p(X,\pi,\dots,\pi),\uc_q(Y,\pi,\dots,\pi)]\\
                             &=\sum_{n=1}^\infty\frac{\hbar^{n-1}}{(n-1)!}\,\uc_n([X,Y],\pi,\dots,\pi)\\
                             &\quad+\sum_{p,q}^\infty[\frac{\hbar^{p-2}}{(p-2)!}\,\uc_p(X,Y,\pi,\dots,\pi),\frac{\hbar^q}{q!}\,\uc_q(\pi,\dots,\pi)].
\end{aligned}
$$
Put $\Phi(X,Y)=\sum_{k=2}^\infty\frac{\hbar^{k-2}}{(k-2)!}\,\uc_k(X,Y,\pi,\dots,\pi)$, then $\Phi(X,Y)\in C^\infty(M)[[\hbar]]$; now we can rewrite the last equality as follows
$$
[\uc_\pi(X),\uc_\pi(Y)]=\uc_\pi([X,Y])+[\Phi(X,Y),B],
$$
or
\begin{equation}
\label{eq:diffcomm}
[\uc_\pi(X),\uc_\pi(Y)]-\uc_\pi([X,Y])=ad_{\Phi(X,Y)},
\end{equation}
where $ad_f,\ f\in C^\infty(M)[[\hbar]]$ is the inner derivative of the deformed algebra $(C^\infty(M)[[\hbar]],\ast)$ with respect to $f$. With a little work we obtain the following easy proposition
\begin{prop}
$\Phi$ determines a class in the $2$-dimensional Lie algebra cohomology of \g\ with values in $C^\infty(M)[[\hbar]]//[C^\infty(M)[[\hbar]],C^\infty(M)[[\hbar]]]$ (here on the right we consider the factor space of $C^\infty(M)[[\hbar]]$ by the subspace of all commutators of its elements with respect to the $\ast$-product). If one can homotopy the $A_\infty$-morphism $\uc$ to $\uc'$ so that $\Phi(X,Y)=0$ for all $X,Y\in\g$, then this class vanishes.
\end{prop}
\begin{rem}\rm
Observe, that this statement gives neither a sufficient, nor a necessary condition for the existence of the Lie algebra representation, verifying the conditions we impose: note that in order to have the desired result, we need only to know that $\mathrm{Im}\,\Phi\in Z(\ac)$, and not that $\Phi=0$. It seems, that more suitable conditions can be found, if we consider more complicated complexes, for example the Chevalley complex with coefficients in Hochschild complex of $(C^\infty(M)[[\hbar]],\ast)$, see next sections for a similar construction.
\end{rem}
\begin{proof} We shall only sketch the proof here. First we check that $d\Phi(X,Y)=0$, where $d$ denotes the Chevalley-Eilenberg differential in the complex with values in $\mathcal C=C^\infty(M)[[\hbar]]/[C^\infty(M)[[\hbar]],C^\infty(M)[[\hbar]]]$. To this end using the definition of $L_\infty$-morphism, similarly to what we did before we compute:
$$
\begin{aligned}
d_{CE}\Phi(X,Y,Z)&=\Phi([X,Y],Z)-\Phi([X,Z],Y)+\Phi([Y,Z],X)\\
                             &\quad-[\uc_\pi(X),\Phi(Y,Z)]+[\uc_\pi(Y),\Phi(X,Z)]-[\uc_\pi(Z),\Phi(X,Y)]=0,
\end{aligned}
$$
where one should use the fact that $\uc_k(X,Y,Z,\pi,\dots,\pi)=0$ due to the dimension restrictions. Observe, that this equality holds in $C^\infty(M)[[\hbar]]$ without passing to the factorspace. On the other hand, since $\uc_\pi(X)$ is a differentiation of $C^\infty(M)[[\hbar]]$, when $X$ is a Poisson field, the space of commutators is preserved by its action; moreover, the equation \eqref{eq:diffcomm} shows, that the action of $\uc_\pi(X)$ descends to an action of \g\ on the factorspace. Thus, the Lie algebra cohomology of \g\ with values in $\mathcal C$  is well defined.

Further, if $\uc$ is homotopic to $\uc'$, then the corresponding $\ast$-products are equivalent, the equivalence being given by the operator $\sum_n\frac{\hbar^n}{n!}\hc_n(\pi,\dots,\pi)$; thus one can identify the corresponding factorspaces. It also follows that the difference of $\uc_\pi(X)$ and $\uc'_\pi(X)$ is equal to the inner derivative with respect to the element $\sum_n\frac{\hbar^n}{n!}\hc_{n+1}(X,\pi,\dots,\pi)$. Thus, the corresponding Chevalley-Eilenberg complexes are equal. Finally, in this case the difference $\Phi(X,Y)-\Phi'(X,Y)$ will be equal to the differential of $\sum_n\frac{\hbar^n}{n!}\hc_{n+1}(X,\pi,\dots,\pi)$, whence the result.
\end{proof}
\begin{rem}\rm
If we take composition of $\Phi$ with any \g-equivariant linear functional on the space $\mathcal C$, we shall obtain a cohomology class with values in $\mathbb R$. A good example of such functional is given by Fedosov's trace (see \cite{Fedosov}), which can be applied to the algebra \ac\ for symplectic manifolds.

Another remark, which we would like to make here is that the space $\mathcal C$ is the $0$-degree part of the Hochschild homology of $(C^\infty(M)[[\hbar]],\ast)$; thus it seems, that this construction should be treated as a part of more general theory which would involve the Hochschild and cyclic cohomology.
\end{rem}
In order to understand the more general situation (when $B$ is not induced by an $A_\infty$-map), and to find finer obstructions for the solution of our problem, we need more accurate considerations, which are given in the following two sections.

\section{The obstruction theory}
\subsection{1d-case}
We begin with the problem of extending a Poisson vector field to a derivative of the deformed algebra. So let $X=X_0$ be a Poisson vector field on $M,\pi$; this is equivalent to the condition that $X$ verifies Leibniz rule with respect to the Poisson bracket $\{,\}$, i.e. 
\begin{equation}
\label{eq:Poivf1}
X(\{f,g\})=\{X(f),g\}+\{f,X(g)\}.
\end{equation}
We are looking for a formal power series operator
$$
\xc= \sum_{k=0}^\infty \hbar^k X_k,
$$
where $X_k: C^\infty(M)\to C^\infty(M),\ k\ge 1$ are some differential operators (and $X_0=X$); it is our purpose to find the series $\xc$ such, that
\begin{equation}
\label{eqsymm}
\xc(f\ast g)=\xc(f)\ast g+ f\ast\xc(g).
\end{equation}
Using the decomposition \eqref{eq:defq1} we can rearrange this graded relation in the form of a series of equations, beginning with:
\begin{align}
\notag
&\qquad\qquad\qquad X_0(fg)-fX_0(g)-X_0(f)g=0,\\
\intertext{which holds, since $X_0$ is a vector field;}
\notag
&\qquad X_1(fg)-fX_1(g)-X_1(f)g=\frac12(\{X_0(f),g\}+\{f,X_0(g)\}-X_0(\{f,g\})),\\
\intertext{which can be easily fulfilled: recall, that $X$ is a differentiation of Poisson bracket (see \eqref{eq:Poivf1}) so the right hand side vanishes; now it is enough to take an arbitrary vector field as $X_1$. Next:}
\label{eqsymdeg2}
&\begin{aligned}
X_2(fg)-fX_2(g)-X_2(f)g&=\frac12(\{X_1(f),g\}+\{f,X_1(g)\}-X_1(\{f,g\}))\\
                                               &\quad+B_2(X_0(f),g)+B_2(f,X_0(g))-X_0(B_2(f,g))
\end{aligned}
\end{align}
The left hand side of this equality is equal to the opposite of Hochschild differential of $X_2$. On the other hand, the expression on the right of this formula can be interpreted as the sum of two Gerstenhaber brackets:
$$
\begin{aligned}
\frac12(\{X_1(f),g\}+\{f,X_1(g)\}-X_1(\{f,g\}))&=[B_1,X_1](f,g),\\
B_2(X_0(f),g)+B_2(f,X_0(g))-X_0(B_2(f,g))&=[B_2,X_0](f,g).
\end{aligned}
$$
So, if we apply Hochschild differental to the right hand side of this formula, we shall obtain:
$$
\begin{aligned}
\delta([B_1,X_1]+[B_2,X_0])&=[\delta(B_1),X_1]-[B_1,\delta(X_1)]\\
                                                                &\quad+[\delta(B_2),X_0]-[B_2,\delta(X_0)]\\
                                                                &=[\delta(B_2),X_0],
\end{aligned}
$$
since $X_0,\,X_1$ and $B_1$ are Hochschild cocycles (the latter follows from the Leibniz rule for a Poisson bracket; c.f. also the first equality in \eqref{eq:MCeq1}). On the other hand, since $\ast$ is an associative product, we have from \eqref{eq:MCeq1}
$$
\delta(B_2)=-\frac12[B_1,B_1].
$$ 
Since $X_0$ is a symmetry of $B_1$ (the latter being given by Poisson bivector), it follows from Jacobi identity that $[[B_1,B_1],X_0]=0$; so the right hand side of the last equation vanishes, and we conclude, that the right hand side is a Hochschild cocycle. 

Recall, that for a Poisson manifld $M,\pi$ its \textit{Lichnerowicz-Poisson} cohomology is defined as the cohomology of the complex $\Gamma(\Lambda^*TM)$ with differential $d_\pi$. Using the identifications of previous section, we can say, that $d_\pi$ is equal to the map in Hochschild cohomology, induced by the Gerstenhaber bracket with $\chi(\pi)$. If we apply this map to the Hochschild cocycle $[B_2,X_0]$, we obtain from Jacobi identity, invariance of $\pi$ with respect to $X=X_0$ and the Maurer-Cartan equation:
$$
[B_1,[B_2,X_0]]=[[B_1,B_2],X_0]=-[\delta B_3,X_0]=\delta(-[B_3,X_0]).
$$
The last equality follows from the fact, that $X_0$ is closed $1$-cochain. Thus the following is true:
\begin{prop}
\label{propdeg1}
One can find the derivative $\xc$ up to the second degree in $\hbar$ iff the class of $[B_2,X_0]\in H^2(C^\infty(M))$ belongs to the image of 
$$
d_\pi:H^1(C^\infty(M))\to H^2(C^\infty(M)),\ \mbox{where}\ d_\pi(Y)=[\pi,Y],
$$
for a vector field $X\in H^1(C^\infty(M))$. Here $\pi$ is the bivector, which defines the Poisson structure and the brackets on the right denote the Schouten brackets on polyvector fields. In other words it vanishes, iff the class of $[B_2,X_0]$ in the Lichnerowicz-Poisson cohomology of $M$ vanishes.
\end{prop}
Now we are going to proceed by induction in the powers of $\hbar$. To make the pattern clear we begin with the next degree: we suppose that $X_0,\,X_1$ and $X_2$ have been chosen so that equation \eqref{eqsymm} holds up to the second degree in $\hbar$. Thus, the first non-zero term, that we should consider is:
\begin{equation}
\label{eqstage3}
X_3(fg)-fX_3(g)-X_3(f)g=[B_1,X_2](f,g)+[B_2,X_1](f,g)+[B_3,X_0](f,g).
\end{equation}
First, we show, that the right hand side of this equation is a Hochschild cocycle. Recall, that by induction hypothesis we have
$$
\begin{aligned}
\delta X_0&=0, & \delta X_1&=0, & \delta X_2&=-[B_1,X_1]-[B_2,X_0],
\end{aligned}
$$
and that
$$
\begin{aligned}
\delta B_1&=0, & \delta B_2&=-\frac12[B_1,B_1], & \delta B_3&=-[B_2,B_1]
\end{aligned}
$$
by Maurer-Cartan equation. So we have, using this and the graded skew symmetry of Gerstenhaber bracket 
$$
\begin{aligned}
\delta([B_1,X_2]+[B_2,X_1]+[B_3,X_0])&=[B_1,[B_1,X_1]]+[B_1,[B_2,X_0]]\\
                                                                         &\quad+\frac12[X_1,[B_1,B_1]]+[X_0,[B_2,B_1]]
\end{aligned}
$$
We claim that the right hand side of this formula vanishes. Using the Jacobi identity for Gerstenhaber bracket and its skew symmetry we have:
$$
\begin{aligned}
 \frac12[X_1,[B_1,B_1]]&+[B_1[B_1,X_1]]=\frac12([X_1,[B_1,B_1]]+2[B_1,[B_1,X_1]])\\
                                                                    &=\frac12([B_1,[B_1,X_1]]-[B_1,[X_1,B_1]+[X_1,[B_1,B_1]])=0
\end{aligned}
$$
Similarly,with the help of Jacobi identity and the fact, that $X_0$ is a symmetry of the Poisson bracket, we have
$$
[B_1,[B_2,X_0]]+[X_0,[B_2,B_1]]\!=\![B_1,[B_2,X_0]]-[B_2,[X_0,B_1]]+[X_0,[B_2,B_1]]=0.
$$
So the claim is true. Further, one can show, that Gerstenhaber bracket of this element with the Poisson bivector is exact with respect to the Hochschild boundary:
$$
\begin{aligned}
{[}B_1,[B_1,X_2]&+[B_2,X_1]+[B_3,X_0]]\\
                          &=[B_1,[B_1,X_2]]+[[B_1,B_2],X_1]-[B_2,[B_1,X_1]]+[[B_1,B_3],X_0]\\
                          &=\frac12[[B_1,B_1],X_2]-\frac12[\delta B_3,X_1]+[B_2,\delta X_2]\\
                          &\quad+[B_2,[B_2,X_0]]-[\delta B_4,X_0]+\frac12[[B_2,B_2],X_0]\\
                          &=-\left([\delta B_2,X_2]-[B_2,\delta X_2]\right)+\frac12\left([\delta B_3,X_1]+[\delta B_4,X_0]\right)\\
                          &\quad+\frac12[[B_2,B_2],X_0]+[B_2,[B_2,X_0]]\\
                          &=\delta([B_2,X_2]+[B_3,X_1]+[B_4,X_0]).
\end{aligned}
$$
Observe, that we can perturb the last chosen element $X_2$ by any vector field $X'$ without spoiling its cohomological properties: this will not change its Hochschild coboundary, so the previous equation \eqref{eqsymdeg2} will not be violated. On the other hand, the element on the right hand side of the equation \eqref{eqstage3} will be perturbed by a Poisson-exact element $d_\pi X$. Thus, we conclude, that the statement of the theorem remains intact: \textit{the existence of $X_3$ depends on the triviality of the class of $[B_1,X_2]+[B_2,X_1]+[B_3,X_0]$ in Poisson cohomology.}

Now the general construction is clear: we begin by supposing that the terms $X_0,\,X_1,\dots,X_n$ have been chosen so, that the equality \eqref{eqsymm} holds up to degree $n$ in $\hbar$. Then the following stage is given by an operator $X_{n+1}$, verifying the equality:
\begin{equation}
\label{eqstagen}
\delta X_{n+1}=-\sum_{k=1}^{n+1}[B_k,X_{n+1-k}].
\end{equation}
Then by inductive hypothesis we have the following properties of $X_k$:
$$
\delta X_k=-\sum_{j=1}^k[B_j,X_{k-j}],
$$
and, since the multiplication is associative
$$
\delta B_k=-\frac12\sum_{i=1}^{k-1}[B_i,B_{k-i}].
$$
Using these two equations, we see that the right hand side of equation \eqref{eqstagen} is a cocycle:
$$
\begin{aligned}
\delta&\left(\sum_{k=1}^{n+1}[B_k, X_{n+1-k}]\right)=\sum_{k=1}^{n+1}\left([\delta B_k,X_{n+1-k}]-[B_k,\delta X_{n+1-k}]\right)\\
          &=-\sum_{k=1}^{n+1}\left(\frac12[\sum_{i=1}^{k-1}[B_i,B_{k-i}],X_{n+1-k}]-[B_k,\sum_{j=1}^{n+1-k}[B_j,X_{n+1-k-j}]]\right)\\
          &=\frac12\sum_{p+q+r=n+1}([X_r,[B_p,B_q]]+2[B_p,[B_q,X_r])=0,
\end{aligned}
$$
where the last equality follows from Jacobi identity. Further, just like in the case of $X_2$ we can reduce the question of finding the extensions $X_{n+1},\ n\ge 2$ to the same form as for $X_2$ and $X_3$. Namely, observe, that adding a vector field $X'$ to $X_n$ does not change the relation, which determines it (since $\delta X'=0$. On the other hand, this perturbation turns the right hand side of equation \eqref{eqstagen} into
$$
[B_1,X']+\sum_{k=1}^{n+1}[B_k,X_{n+1-k}].
$$
Both terms, as we know, are closed Hochschild cochains, and the first one (after passing to cohomology) has the form $d_\pi(X')$, where $d_\pi$ is Lichnerowicz's Poisson cohomology differential. Thus, we conclude:
\begin{prop}
One can find a continuation $X_{n+1}$ of the deformed symmetry, if and only if the right hand side of equality \eqref{eqstagen} gives a trivial element in Lichnerowicz's Poisson cohomology.
\end{prop}
To prove this, we need just to show, that the element on the right is closed with respect to $d_\pi$, when we pass to cohomology. But this follows easily from the relations (modulo exact Hochschild cochains):
$$
[B_1,X_n]=\sum_{k=2}^{n+1}[B_k,X_{n+1-k}]
$$
and 
$$
[B_1,B_n]=\frac12\sum_{i=2}^{k-2}[B_i,B_{k-i}].
$$
These are just the relations we gave earlier, where we omit the Hochschild differential (since it in any case shall vanish on the level of cohomology).
\begin{rem}
\label{rem:partialcase}
Observe, that it follows from the results of the previous section, that \textit{given a Poisson vector field $X$, one can always find a sequence of operators $X=X_0,X_1,X_2,\dots$ so that the operator \xc\ will be a differentiation of $(\ac,\ast)$, i.e. so that all the obstructions we listed here will vanish.} However, this does not mean, that the obstructions we consider here are not necessary at all: in fact, they answer the question, whether the given set of operators can be considered as the first stage of a differentiation of \ac.

Also this approach will turn fruitful in the next section.
\end{rem}
\subsection{General case}
Let $\g$ be a Lie algebra, acting on a Poisson manifold $M$, i.e. represented in the Lie algebra $ D^1_\pi(C^\infty(M))$ of Poisson vector fields on $M$ that is vector fields, commuting with the Poisson bivector $\pi$. The question is: is it possible to extend this representation to a representation of $\g$ by derivations of the quantized algebra? In this section we assume, that the bidifferential operators $B_{2k-1}$ are antisymmetric and $B_{2k},\ k\ge 1$ are symmetric  (in particular, this is the case of the operators, constructed by Kontsevich's formula, see \cite{Kon97}).

In order to answer this question, we consider this map in a generic form 
$$
\Phi=\varphi_0+\hbar\varphi_1+\hbar^2\varphi_2+\dots:\g\to C^1_{loc}(C^\infty(M))[[\hbar]].
$$
We need to find $\Phi$ such that the conditions above would hold, i.e. that it is a representation of $\g$ in derivations of $(C^\infty(M)[[\hbar]],\ast)$. Just like in the previous section, one can start reasoning inductively: we assume, that the $0$-degree part of $\Phi$ is given by a representation $\varphi_0:\g\to Vect_\pi(M)$ of \g\ in the Lie algebra of Poisson vector fields on $M,\pi$. It is clear, that this map verifies both conditions (i.e. that its image consist of derivations of the deformed algebra and that it is a representation of \g) up to degree $1$ in parameter $\hbar$. Then we look for a ``correction term'' $\varphi_1:\g\to C^1_{loc}(C^\infty(M))$; the map $\varphi_1$ should be such, that the sum $\varphi_0+\hbar\varphi_1$ verifies the above mentioned conditions up to degree $2$ in $\hbar$. So when we restrict our attention to the degrees less than, or equal to $2$ in $\hbar$, we obtain the following two equalities:
 $$
\delta\varphi_1(\xi)=\delta\varphi_1(\eta)=0,\ [\varphi_1(\xi),\varphi_0(\eta)]+[\varphi_0(\xi),\varphi_1(\eta)]-\varphi_1([\xi,\eta])=0
 $$
for all elements $\xi,\eta\in\g$. Here, as before, $[,]$ denotes the Gerstenhaber brackets. It follows from the first equality, that $\varphi_1$ should take values in Hochschild cocycles. Similarly, the left hand side of the second equation here is equal to the Chevalley differential $\partial_\g(\varphi_1)(\xi,\eta)$ of the map $\varphi_1$ viewed as an element of Chevalley complex of \g\ with values in the complex of Hochschild cochains on which \g\ acts via the representation $\varphi_0$. Thus, the first stage of deformation can be achieved by choosing an arbitrary $1$-cocycle in the complex $C^*(\g,\,C^*(C^\infty(M)))$ (zero cocycle can also be a choice).

Now, the next stage gives the following equations on the element $\varphi_2$, the next term in the series $\varphi_0+\hbar\varphi_1+\hbar^2\varphi_2+\dots$:
$$
\begin{aligned}
\delta(\varphi_2(\xi))&=-[B_1,\varphi_1(\xi)]-[B_2,\varphi_0(\xi)],\\
\partial_\g(\varphi_2)(\xi,\eta)&=[\varphi_1(\xi),\varphi_1(\eta)].
\end{aligned}
$$
Once again, this equalities should hold for any $\xi,\,\eta\in\g$. The right hand side of the first equation is closed with respect to the Hochschild differential $\delta$ (this can be proved by the same calculation as above). It is also closed with respect to the Chevalley differential $\partial_\g$: we put
$$
\omega_2^1=[B_1,\varphi_1]+[B_2,\varphi_0]:\g\to C^2(C^\infty(M)),
$$
then we compute
$$
\begin{aligned}
\partial_\g(\omega_2^1)(\xi,\eta)&=\varphi_0(\xi)([B_1,\varphi_1(\eta)]+[B_2,\varphi_0(\eta)])-\varphi_0(\eta)([B_1,\varphi_1(\xi)]\\
                                               &\quad+[B_2,\varphi_0(\xi)])-[B_1,\varphi_1([\xi,\eta])]-[B_2,\varphi_0([\xi,\eta])]\\
                                               &=[\varphi_0(\xi),[B_1,\varphi_1(\eta)]+[\varphi_0(\xi),[B_2,\varphi_0(\eta)]]-[\varphi_0(\eta),[B_1,\varphi_1(\xi)]]\\
                                               &\quad-[\varphi_0(\eta),[B_2,\varphi_0(\xi)]]-[B_1,[\varphi_1(\xi),\varphi_0(\eta)]]\\
                                               &\quad-[B_1,[\varphi_0(\xi),\varphi_1(\eta)]]-[B_2,[\varphi_0(\xi),\varphi_0(\eta)]]=0.
\end{aligned}
$$
The last equality here follows from the skew-antisymmetry and Jacobi identity. The right hand side of the second equality (which we denote as $\omega_2^2$) is clearly closed with respect to the Hochschild differential; Chevalley differential $\partial_\g$, applied to it gives:
$$
\begin{aligned}
\partial_\g(\omega_2^1)(\xi,\eta,\zeta)&=[\varphi_0(\xi),[\varphi_1(\eta),\varphi_1(\zeta)]]-[\varphi_0(\eta),[\varphi_1(\xi),\varphi_1(\zeta)]]\\
                                                                       &\quad+[\varphi_0,(\zeta)[\varphi_1(\xi),\varphi_1(\eta)]]+[\varphi_1([\xi,\eta]),\varphi_1(\zeta)]\\
                                                                       &\quad-[\varphi_1([\xi,\zeta]),\varphi_1(\eta)]+[\varphi_1([\eta,\zeta]),\varphi_1(\xi)],
\end{aligned}
$$
which is equal to $0$, because of the Jacobi identity and the assumption, that $\varphi_1$ is a Chevalley cocycle. Thus, the sum $\omega_2^1+\omega_2^2$ is a closed element in the bicomplex $C^*(\g,C^*(C^\infty(M)))$, the Chevalley complex of $\g$ with coefficients in the Hochschild complex of $C^\infty(M)$. In order to be able to find $\varphi_2$ we must choose $\varphi_1$ so, that the cohomology class of this element were equal to $0$. To this end we can vary $\varphi_1$ a little bit so, that it would remain closed with respect to both Hochschild and Chevalley differentials (i.e. so that the previous conditions still hold). The first condition means, that we can only add a Chevalley 1-cochain on $\g$ with values in vector fields on $M$, while the second condition says, that this correction term should be closed with respect to $\partial_g$. In other words, we can add to $\varphi_1$ an arbitrary Chevalley $1$-cocycle $\psi:\g\to Vect(M)$.

This correction term changes the first equation for $\varphi_2$  by adding a new term of the form $[B_1,\psi(\xi)]$. When we pass to Hochschild homology, this term will turn into the Lichnerowicz's Poisson cohomology differential. Thus, we can interpret the first equation as follows: consider the double complex $C^*(\g,CP^*(M))$, i.e. the Chevalley complex of \g\ with coefficients in the Lichnerowicz's complex of $M$. Then the element $[B_2,\varphi_0]$ in $C^1(\g,C^2(C^\infty(M)))$ is closed with respect to both differentials (to see this, just observe, that the terms in its Chevalley and Hochschild differentials above kill each other, and do not interfere with the differentials of $[B_1,\varphi_1]$), in particular, with respect to the Hochschild differential. Thus, it induces an element in $C^1(\g,CP^2(M))$, closed with respect to the Chevalley differential. An easy calculation, similar to the computations from the previous sections, shows that $d_\pi$ vanishes on it too. Thus, it gives an element $\tilde\omega_2$ in the bicomplex cohomology, i.e. in $H^3(\g,CP^*(M))$. Then $\tilde\omega_2$ is equal to zero, iff one can find an element $c=c^0+c^1+c^2$ in $C^0(\g,CP^2(M))\oplus C^1(\g,CP^1(M))\oplus C^2(\g,CP^0(M))$, such that $\partial_\g c+d_\pi c=[B_2,\varphi_0]$. Comparing the bidegrees on both sides, we see, that
$$
d_\pi c^0=0,\ \partial_\g c^0+d_\pi c^1=[B_2,\varphi_0],\ \partial_\g c^1+d_\pi c^2=0\ \mbox{and}\ \partial_\g c^2=0.
$$
This is a bit less than what one should look for: in fact, we need $c^0=c^2=0$. In this case we would have
$$
d_\pi c^1= [B_2,\varphi_0],\ \partial_\g c^1=0,
$$
where $c^1$ is a Chevalley $1$-cochain on \g\ with values in vector fields on $M$. Taking $\varphi_1=-c^1$, we conclude, that the Hochschild class of $[B_1,\varphi_1]+[B_2,\varphi_0]$ is equal to zero in this case, hence we can find $\varphi_2$, verifying the equation 
\begin{equation}
\label{ext21}
\delta(\varphi_2(\xi))=[B_1,\varphi_1(\xi)]+[B_2,\varphi_0(\xi)].
\end{equation}
Thus, we should consider the sub-bicomplex 
$$
\tilde C^*(\g,CP^*(M))=\bigoplus_{p,q>1}C^p(\g,CP^q(M))\subseteq C^*(\g,CP^*(M)).
$$
We conclude, that \textit{there exists an extension $\varphi_2$, verifying the equality \eqref{ext21}, iff the class of $[B_2,\varphi_0]$ in the cohomology of $\tilde C^*(\g,CP^*(M))$ is equal to $0$.}

Let us now suppose, that the equation \eqref{ext21} holds and consider the second equality on $\varphi_2$ i.e.
\begin{equation}
\label{ext22}
\partial_\g(\varphi_2)(\xi,\eta)=[\varphi_1(\xi),\varphi_1(\eta)].
\end{equation}
It is easy to see, that the expression on the right hand side is closed with respect to the Hochschild differential. On the other hand, if we apply Hochschild differential to the left hand side, we shall get $0$, because
$$
\delta(\partial_\g(\varphi_2))=-\partial_\g(\delta\varphi_2)=-\partial_\g([B_1,\varphi_1(\xi)]+[B_2,\varphi_0(\xi)])=0.
$$
Thus, we can pass to the Hochschild cohomology on both sides. Consider the corresponding element in $C^*(\g,CP^*(M))$ (i.e. the difference between the cohomology classes from the left and the right side of equation \eqref{ext22}). Arguing just like in the previous section, one can show that it is closed with respect to both differentials of this complex. On the other hand, we cannot change $\varphi_1$ otherwise, but by adding a $d_\pi$-closed $1$-cocycle on \g, if we don't want to spoil the equality \eqref{ext21}. For example such correction cocycle can be given by the formula $d_\pi(f(\xi))$, where $f:\g\to C^\infty(M)$ is a $C^\infty(M)$-valued $1$-dimensional \g-cocycle (i.e. the value of this map will be in the space of Hamiltonian vector fields on $M$); if the Poisson structure we use is in fact symplextic, than this is (locally) a unique choice. This operation will not change the cohomology class of the element in $C^*(\g,CP^*(M))$ since modulo closed (with respect to Hochschild differential) elements we have
$$
[d_\pi f(\xi),\varphi_1(\eta)]=[[B_1,f(\xi)],\varphi_1(\eta)]=-[B_1,\varphi_1(\eta)(f(\xi))]+[f(\xi),[B_2,\varphi_0(\eta)]],
$$
where the last term is equal to $0$, since $B_2$, and hence $[B_2,\varphi_0(\eta)]$ is symmetric bidifferential operator; so
$$
[B_2,\varphi_0(\eta)](f(\xi),g)-[B_2,\varphi_0(\eta)](g,f(\xi))=0
$$
for all $g$. Similarly, we can change $\varphi_2$ only by a $1$-cochain $c:\g\to CP^1(M)$, i.e. by a cochain with values in vector fields (so that the Hochschild differential of $\varphi_2$ remains unchanged).

Thus, we conclude, that the question, whether it is possible to choose $\varphi_2$ verifying \eqref{ext21} so that the condition \eqref{ext22} holds, can be reduced to the following: choose arbitrary $\varphi_2$, verifying \eqref{ext21}, then consider the difference $\partial_\g\varphi_2-[\varphi_1,\varphi_1]$ as an element in $H^2(\g,CP^1(M))$. If this element is trivial, then we can further change $\varphi_2$ as needed.

Now, we can pass in a similar way to the case $n=3$: then we have the following two equations
\begin{align}
\label{ext31}
\delta\varphi_3(\xi)&=[B_1,\varphi_2(\xi)]+[B_2,\varphi_1(\xi)]+[B_3,\varphi_0(\xi)]\\
\label{ext32}
(\partial_\g\varphi_3)(\xi,\eta)&=[\varphi_2(\xi),\varphi_1(\eta)]+[\varphi_2(\eta),\varphi_1(\xi)].
\end{align}

Now we want to make \eqref{ext31} hold without disrupting \eqref{ext21} and \eqref{ext22}. This means, that we can change $\varphi_2$ only by adding to it a closed $1$ \g-cochain with values in $CP^1(M)$. On the other hand, reasoning as above, we see, that the right hand side of equation \eqref{ext31} is closed with respect to the Hochschild differential $\delta$ and (when we pass to the cohomology) with respect to the Poisson differential $d_\pi$ and Chevalley differential $\partial_\g$. Thus, as before we conclude: \textit{one can choose $\varphi_3$, so that the equality \eqref{ext31} would hold, if the class of the right hand side of this equation in the cohomology of bicomplex $\tilde C^*(\g,CP^*(M))$ vanishes.}

Further, as before, changing $\varphi_3$ by a $CP^1(M)$-valued $1$ \g-cochain, we see, that \textit{one can choose $\varphi_3$ so, that \eqref{ext32} would hold, if the class of the difference $\partial_\g\varphi_3(\xi,\eta)-[\varphi_2(\xi),\varphi_1(\eta)]+[\varphi_2(\eta),\varphi_1(\xi)]$ in Chevalley cohomology $H^2(\g,CP^1(M))$ is trivial.}

Finally, reasoning by induction we obtain the following general statement:
\begin{prop}
Suppose, that we have found the maps $\varphi_1,\varphi_2,\dots,\varphi_n$ so that the conditions on $\Phi_n=\varphi_0+\hbar\varphi_1+\dots+\hbar^n\varphi_n$ hold up to $\hbar^n$. The one can choose $\varphi_{n+1}$, so that for the map $\Phi_n+\hbar^{n+1}\varphi_{n+1}$ the first condition (i.e. that this map is derivation) would hold up to degree $n+1$ in $\hbar$, if the class of 
$$
\omega'_n(\xi)=[B_1,\varphi_n(\xi)]+[B_2,\varphi_{n-1}(\xi)]+\dots+[B_{n+1},\varphi_0(\xi)]
$$
in the cohomology of bicomplex $\tilde C^*(\g,CP^*(M))$ vanishes. Further, one can choose this same $\varphi_{n+1}$ so, that the second condition (i.e. that that this map is a representation of \g) would also hold up to $\hbar^{n+1}$, if the class of the element 
$$
\begin{aligned}
\omega''_n&=\partial_\g\varphi_{n+1}(\xi,\eta)-[\varphi_n(\xi),\varphi_1(\eta)]-[\varphi_{n-1}(\xi),\varphi_2(\eta)]-\dots-[\varphi_(\xi),\varphi_1(\eta)]\\
                  &\quad+[\varphi_n(\eta),\varphi_1(\xi)]-[\varphi_{n-1}(\eta),\varphi_2(\xi)]-\dots-[\varphi_(\eta),\varphi_1(\xi)]
\end{aligned}
$$
in Chevalley cohomology $H^2(\g,CP^1(M))$ is trivial.
\end{prop}

\medskip
The author would like to express his gratitude for hospitality and wonderful working conditions to the University of Angers, and VIASM, where part of this work was completed. I would also like to thank Andrey Konyaev and Vladimir Roubtsov for numerous fruitful discussions.

\end{document}